\documentclass{article}

\usepackage{amsmath,amssymb}
\usepackage{amscd}
\usepackage{mathrsfs}
\usepackage{amsthm}

\theoremstyle{plain}

\newtheorem{theorem}{\bf Theorem}[section]
\newtheorem{lemma}[theorem]{\bf Lemma}

\theoremstyle{definition}
\newtheorem{definition}[theorem]{\bf Definition}
\newtheorem{example}[theorem]{\bf Example}
\newtheorem{remark}[theorem]{\bf Remark}

\newcommand{\eqa}[1]{
\begin{align*}
#1
\end{align*}}

\newcommand{\nai}[2]{\langle #1,#2\rangle}
\newcommand{\dom}[1]{{{\rm{dom}}{(#1)}}}

\title{Notes on the Krupa-Zawisza Ultrapower of Self-Adjoint Operators}

\author{Hiroshi Ando \and Izumi Ojima \and Hayato Saigo}

\begin{document}
\maketitle
\begin{abstract}
Let $\omega\in \beta \mathbb{N}\setminus \mathbb{N}$ be a free ultrafilter on $\mathbb{N}$. It is known that there is a difficulty in constructing the ultrapower of unbounded operators. Krupa and Zawisza gave a rigorous definition of the ultrapower $A_{\omega}$ of a self-adjoint operator $A$. In this note, we give alternative description of $A_{\omega}$ and the Hilbert space $H(A)$ on which $A_{\omega}$ is densely defined, which provides a criterion to determine to which representing sequence $(\xi_n)_n$ of a given vector $\xi\in \text{dom}(A_{\omega})$ has the property that $A_{\omega}\xi=(A\xi_n)_{\omega}$ holds. An explicit core for $A_{\omega}$ is also described.
\end{abstract}


\medskip


\medskip
\section{Introduction}
Throughout the paper, we fix a free ultrafilter $\omega$ on $\mathbb{N}$ and a separable infinite-dimensional Hilbert space $H$. We denote by $\mathbb{B}(H)$ the algebra of all bounded operators in $H$. Let $H_{\omega}$ be the Hilbert space ultraproduct of $H$. 
Each bounded sequence $(a_n)_n\subset \mathbb{B}(H)$ of bounded operators in $H$ defines a bounded operator $(a_n)_{\omega}\in \mathbb{B}(H)$, called the ultraproduct of $(a_n)_n$ by the formula
\[(a_n)_{\omega}(\xi_n)_{\omega}:=(a_n\xi_n)_{\omega}, \ \ \ \ \ (\xi_n)_{\omega}\in H_{\omega}.\]
The ultrapower (or more generally the ultraproduct) of a sequence of bounded operators has been used as an efficient tool for the analysis on Hilbert spaces. 
In view of its usefullness, it is natural to consider a corresponding notion of ultrapower $A_{\omega}$ for unbounded self-adjoint operator $A$. However, there arises essential difficulties on the definition:
\begin{itemize}
\item[(1)] Definition of the domain $\dom{A_{\omega}}$ of $A_{\omega}$. 
\item[(2)] Self-adjointness of $A_{\omega}$.
\item[(3)] Interpretation of $A_{\omega}(\xi_n)_{\omega}=(A\xi_n)_{\omega}$ for $\xi=(\xi_n)_{\omega}\in \dom{A_{\omega}}.$
\end{itemize}

Regarding (1), it does not makes sense to define $\dom{A_{\omega}}$ to be the subspace $\dom{A}_{\omega}$ of all $\xi\in H_{\omega}$ which is represented by a sequence $(\xi_n)_n$ where $\xi_n\in \dom{A}$ for all $n$, because $\dom{A}_{\omega}$ is simply the whole $H_{\omega}$ and the definition $A_{\omega}(\xi_n)_{\omega}=(A\xi_n)_{\omega}$ is not well-defined. Importance of the question (2) should be clear. (3) is probably the most delicate problem. Even if we could manage to define $\dom{A_{\omega}}$ and suppose $\xi \in \dom{A_{\omega}}$ is represented by $(\xi_n)_n$ with $\xi_n\in \dom{A}$ for all $n$, it might be the case that there exists another $(\xi_n')_n$ which also represents $\xi$ (i.e., $\lim_{n\to \omega}\|\xi_n-\xi_n'\|=0$ holds), and $\xi_n'\in \dom{A}$ for all $n$ holds as well, and yet $(A\xi_n)_{\omega}\neq (A\xi_n')_{\omega}$.
\begin{example}\label{ex: (Axi_n) not equals (Axi_n')}
Let $A$ be a self-adjoint operator and assume that there is an orthonormal base $\{\eta_n\}_{n=1}^{\infty}$ of $H$ consisting of eigenvectors of $A$ with $A\eta_n=n\eta_n,\ \ n\ge 1$. Let $\eta \in \dom{A}$, and consider two sequences 
\[\xi_n:=\eta,\ \ \xi_n':=\eta+\frac{1}{n}\eta_n \ \ \ (n\ge 1).\]
Then it is clear that $\xi_n,\xi_n'\in \dom{A}$, that $(\xi_n)_n, (\xi_n')_n$ defines the same element $\xi=(\xi_n)_{\omega}=(\xi_n')_{\omega}\in H_{\omega}$, but 
\[\lim_{n\to \omega}\|A\xi_n-A\xi_n'\|=\lim_{n\to \omega}\|\eta_n\|=1\neq 0,\]
whence $(A\xi_n)_{\omega}\neq (A\xi_n')_{\omega}$. Should we define $A_{\omega}\xi=(A\xi_n)_{\omega}$ or $A_{\omega}\xi=(A\xi_n')_{\omega}$? 
\end{example}
Despite the above difficulty, Krupa and Zawisza \cite{KZ} gave a rigorous definition of $A_{\omega}$, as well as interesting applications to Schr\"odinger operators. To define $\dom{A_{\omega}}$ in any sensible way, it is necessary to note that such a domain must be in the subspace of $\mathscr{D}_A$, given as the set of all $\xi\in H_{\omega}$ which has a representing sequence $(\xi_n)_n$ of vectors from $\dom{A}$ such that $(A\xi_n)_n$ is also norm-bounded. We put $H(A)=\overline{\mathscr{D}_A}$. 
We recall from \cite{KZ} the notion of partial ultrapowers.
\begin{definition} 
Let $\mathcal{H}\subset H_{\omega}$ be a closed subspace. A densely defined operator $\mathscr{A}$ in $\mathcal{H}$ is called a {\it partial ultrapower} (p.u. for short) of $A$ in $\mathcal{H}$, if for any $\xi \in \dom{\mathscr{A}}$ there is $(\xi_n)_n\subset \dom{A}$ such that $\xi=(\xi_n)_{\omega}$ and $\mathscr{A}\xi=(A\xi_n)_{\omega}.$
\end{definition}
(One of) the fundamental results of \cite{KZ} is 
\begin{theorem}\label{thm: KZ p.u.}
\begin{itemize}
\item[{\rm{(1)}}] there is a p.u. $A_{\omega}$ of $A$ in $H(A)$ satisfying $\dom{A_{\omega}}=\mathscr{D}_A$, uniquely determined by the property that for $\xi\in \mathscr{D}_A$ and $\eta\in H(A)$, 
$A_{\omega}\xi=\eta$ if and only if there is a representative $(\xi_n)_n\subset \dom{A}$ of $\xi$ satisfying $(A\xi_n)_{\omega}=\eta$. 
\item[{\rm{(2)}}] $A_{\omega}$ is the maximal among all p.u.'s of $A$. I.e., if $\mathscr{A}$ is a p.u. of $A$ in $\mathcal{H}$, then $\mathcal{H}\subset H(A)$ and $\mathscr{A}=A_{\omega}|_{\dom{\mathscr{A}}}$.
\item[{\rm{(3)}}] $A_{\omega}$ is self-adjoint in $H(A)$. Moreover, $(A_{\omega}-i)^{-1}$ is the restriction of $((A-i)^{-1})_{\omega}$ to $H(A)$ and ${\rm{sp}}(A_{\omega})={\rm{sp}}(A)$ holds.  
\end{itemize}
\end{theorem}
Note that in (1), the uniqueness of $\eta$ is guaranteed by the condition $\eta\in H(A)$: in fact in Example \ref{ex: (Axi_n) not equals (Axi_n')}, $(A\xi_n)_{\omega}\in H(A)$, while $(A\xi_n')_{\omega}\notin H(A)$ (see Remark \ref{rem: (Axi_n') is bad}). Despite their success, what seems to be unsatisfactory is that there is no a priori criterion for a given $\xi\in \mathscr{D}_A$ to choose an appropriate representative $(\xi_n)_n$ such that $(A\xi_n)_{\omega}$ is well-defined and is in $H(A)$. Whether a chosen representative is indeed appropriate or not, can be seen only after one applies $A$ and to know that the resulting vector is in the closure of $\mathscr{D}_A$. In this short note, we give an alternative characterization of such an appropriate sequence, which we call a proper $A$-sequence, and give a new description of $A_{\omega}$ in terms of an auxiliary operator $\widetilde{A}_{\omega}$ by checking the validity of the equality $A_{\omega}=\widetilde{A}_{\omega}$. More precisely, we show that a bounded sequence $(\xi_n)_n$ of vectors from $\dom{A}$ has a property that $A_{\omega}(\xi_n)_{\omega}=(A\xi_n)_{\omega}$, if and only if $(A\xi_n)_n$ is bounded, and for every $\varepsilon>0$, there is $a>0$, $(\eta_n)_n\in \ell^{\infty}(\mathbb{N},H)$ with $\eta_n\in 1_{[-a,a]}(A)H$ for each $n\in \mathbb{N}$, such that $\lim_{n\to \omega}\|\xi_n-\eta_n\|_A<\varepsilon$. $(\|\cdot \|_A$ is the graph norm). Moreover, a bounded sequence $(\xi_n)_n$ defines an element in $H(A)$, if and only if the family of maps $\{f_n\colon \mathbb{R}\to H\}_{n=1}^{\infty}$ given by $f_n(t)=e^{itA}\xi_n$, is $\omega$-equicontinuous (see Definition \ref{def: omega-equicontinuous}). We believe that this description will make Krupa-Zawisza analysis more accessible and give new insights to them.    
\section{Preliminaries}

Let $\ell^{\infty}(\mathbb{N},H)$ be the space of bounded sequences in $H$. The ultrapower $H_{\omega}$ of $H$ is defined by 
$H_{\omega}=\ell^{\infty}(\mathbb{N},H)/\mathcal{T}_{\omega}$, where $\mathcal{T}_{\omega}$ is the subspace of $\ell^{\infty}(\mathbb{N},H)$ consisting of sequences tending to 0 in norm along $\omega$. The canonical image of $(\xi_n)_n\in \ell^{\infty}(\mathbb{N},H)$ is written as $(\xi_n)_{\omega}$, and $H_{\omega}$ is again a Hilbert space (non-separable in general) by the inner product
\[\nai{\xi}{\eta}=\lim_{n\to \omega}\nai{\xi_n}{\eta_n},\ \ \ \ \ \xi=(\xi_n)_{\omega},\ \eta=(\eta_n)_{\omega}\in H_{\omega}.\]
We identify $\xi \in H$ with its canonical image $(\xi,\xi,\cdots )_{\omega}\in H_{\omega}$, so that $H$ is a closed subspace of $H_{\omega}$. 
Let $\{a_n\}_{n=1}^{\infty}$ be a sequence of bounded operators on $H$ with $\sup_n \|a_n\|<\infty$. We then define a bounded operator $(a_n)_{\omega}\in \mathbb{B}(H_{\omega})$ by
\[(a_n)_{\omega}(\xi_n)_{\omega}:=(a_n\xi_n)_{\omega},\ \ \ \ (\xi_n)_{\omega}\in H_{\omega}.\]
$(a_n)_{\omega}$ is well-defined by the above, and $\|(a_n)_{\omega}\|=\lim_{n\to \omega}\|a_n\|$ holds. 
For a linear operator $T$ on $H$, the domain of $T$ is denoted as $\dom{T}$. For $\xi\in \dom{T}$, we denote $\|\xi\|_T$ the graph norm of $T$ given by $(\|\xi\|^2+\|T\xi\|^2)^{\frac{1}{2}}$.  For details about operator theory, see e.g., \cite{Sch}.  
\section{Construction of $\widetilde{A}_{\omega}$}
Let $A$ be a self-adjoint operator on a separable Hilbert space $H$, and let $u(t)=e^{itA} (t\in \mathbb{R})$. We introduce several subspaces of $H_{\omega}$. First, we need to introduce the notion of $\omega$-equicontinuity which has been used in the literature (see \cite{Ki,MaTo}).
\begin{definition}\label{def: omega-equicontinuous}
Let $(X_1,d_1),\ (X_2,d_2)$ be metric spaces. A family of maps $\{f_n\colon X_1\to X_2\}_{n=1}^{\infty}$ is said to be $\omega$-{\it equicontiuous} if for every $x\in X$ and $\varepsilon>0$, there exists $\delta=\delta_{x,\varepsilon}>0$ and $W\in \omega$ such that for every $x'\in X$ with $d_1(x,x')<\delta$ and $n\in W$, we have 
\[d_2(f_n(x),f_n(x'))<\varepsilon.\] 
\end{definition}
\begin{lemma}\label{lem: omega-equicontinuity is well-defined} Let $(\xi_n)_n\in \ell^{\infty}(\mathbb{N},H)$ be a sequence such that  $\{f_n:t\mapsto e^{itA}\xi_n\}_{n=1}^{\infty}$ is $\omega$-equicontinuous. Then $t\mapsto (e^{itA}\xi_n)_{\omega}$ is continuous. Moreover, if $(\xi_n')_n\in \ell^{\infty}(\mathbb{N},H)$ satisfies $\lim_{n\to \omega}\|\xi_n-\xi_n'\|=0$, then $\{f_n':t\mapsto e^{itA}\xi_n'\}_{n=1}^{\infty}$ is also $\omega$-equicontinuous.  
\end{lemma}

\begin{proof}
Let $t\in \mathbb{R}, \varepsilon>0$ be given. There exists $\delta>0$ and $W_1\in \omega$ such that for any $s\in (t-\delta,t+\delta)$ and $n\in W_1$, we have $\|e^{itA}\xi_n-e^{isA}\xi_n\|<\varepsilon/3$. This means that $\|(e^{itA}\xi_n)_{\omega}-(e^{isA}\xi_n)_{\omega}\|<\varepsilon/3$, whence $t\mapsto (e^{itA}\xi_n)_{\omega}$ is continuous. By $(\xi_n)_{\omega}=(\xi_n')_{\omega}$, $W_2:=\{n\in \mathbb{N}; \|\xi_n-\xi_n'\|<\varepsilon/3\}\in \omega$. Then for $s\in (t-\delta,t+\delta)$ and $n\in W:=W_1\cap W_2\in \omega$, we have 
\eqa{
\|e^{itA}\xi_n'-e^{isA}\xi_n'\|&\le \|e^{itA}(\xi_n'-\xi_n)\|+\|e^{itA}\xi_n-e^{isA}\xi_n\|+\|e^{isA}(\xi_n-\xi_n')\|\\
&<\varepsilon.
}
Therefore $\{t\mapsto e^{itA}\xi_n'\}_{n=1}^{\infty}$ is $\omega$-equicontinuous. 
\end{proof}

\begin{definition}
A vector $\xi=(\xi_n)_{\omega}\in H_{\omega}$ is called $A${\it -regular} if $\{t\mapsto e^{itA}\xi_n\}_{n=1}^{\infty}$ is $\omega$-equicontinuous. 
By Lemma \ref{lem: omega-equicontinuity is well-defined}, this notion does not depend on the choice of the representing sequence $(\xi_n)_n$. 
\end{definition}

\begin{definition}Under the above notation, we define:
\begin{list}{}{}
\item[(1)] Let $K(A)$ be the set of all $A$-regular vectors of $H_{\omega}$.
\item[(2)] Let $\dom{\widetilde{A}_{\omega}}$ be the set of $\xi \in K(A)$ for which $\lim_{t\to 0}\frac{1}{t}(u(t)_{\omega}-1)\xi$ exists.  
\end{list}
\end{definition}
\begin{lemma}\label{lem: K is closed}
$K(A)$ is a closed subspace of $H_{\omega}$ invariant under $u(t)_{\omega}$ for all $t\in \mathbb{R}$.
\end{lemma}
\begin{proof}
It is clear that $K(A)$ is a subspace of $H_{\omega}$, and that $K(A)$ is $u(t)_{\omega}$-invariant for all $t\in \mathbb{R}$. 
Let $\xi=(\xi_n)_{\omega} \in \overline{K(A)}$ and $\varepsilon>0$. There exists $\eta =(\eta_n)_{\omega}\in K(A)$ such that $\|\xi-\eta\|<\varepsilon/3$. 
Let $t\in \mathbb{R}$. By the $\omega$-equicontinuity of $\{f_n:t\mapsto e^{itA}\eta_n\}_{n=1}^{\infty}$, there exists $\delta>0$ and $W_1\in \omega$ such that for each $s\in (t-\delta,t+\delta)$ and $n\in W_1$, we have $\|e^{itA}\eta_n-e^{isA}\eta_n\|<\varepsilon/3$. Let $W_2:=\{n\in \mathbb{N}; \|\xi_n-\eta_n\|<\varepsilon/3\}\in \omega$. 
Then we have for $s\in (t-\delta,t+\delta)$ and $n\in W:=W_1\cap W_2\in \omega$, that
\eqa{
\|e^{itA}\xi_n-e^{isA}\xi_n\|&\le \|e^{itA}(\xi_n-\eta_n)\|+\|e^{itA}\eta_n-e^{isA}\eta_n\|+\|e^{isA}(\eta_n-\xi_n)\|\\
&<\varepsilon.
}

Therefore $\xi=(\xi_n)_{\omega}$ is $A$-regular, and $\xi \in K(A)$. 
\end{proof}
By Lemma \ref{lem: K is closed}, $v(t):=u(t)_{\omega}|_{K(A)}$ is a continuous one-parameter unitary group of $K(A)$. Therefore by Stone Theorem, there exists a self-adjoint operator $\widetilde{A}_{\omega}$ with domain $\dom{\widetilde{A}_{\omega}}$ such that 
\[i\widetilde{A}_{\omega}\xi=\lim_{t\to 0}\frac{1}{t}(v(t)-1)\xi,\ \ \ \ \xi \in \dom{\widetilde{A}_{\omega}}.\]
In the sequel, we will show that $\widetilde{A}_{\omega}\xi=(A\xi_n)_{\omega}$ for appropriate $(\xi_n)_n$ representing $\xi \in \dom{\widetilde{A}_{\omega}}$.  
\begin{definition}\label{def: proper A-sequences}Let $A$ be a self-adjoint operator on $H$.\\
(1)\ A sequence $(\xi_n)_n\in \ell^{\infty}(\mathbb{N},H)$ is called an {\it $A$-sequence} if $\xi_n\in \dom{A}$ for all $n\ge \mathbb{N}$.\ We denote the space of $A$-sequences as $\ell^{\infty}(\mathbb{N},\dom{A})$.\\
(2) An $A$-sequence $(\xi_n)_n$ is called {\it proper}, if it satisfies the following condition:\\
\ \ \ (*) For each $\varepsilon>0$, there exists $a>0$ and an $A$-sequence $(\eta_n)_n$ with the following properties:
\begin{itemize}
\item[{\rm{(i)}}] $\eta_n\in 1_{[-a,a]}(A)H$ for all $n\ge 1$.
\item[{\rm{(ii)}}] $(A\xi_n)_n\in \ell^{\infty}(\mathbb{N},H)$, and $\displaystyle \lim_{n\to \omega}\|\xi_n-\eta_n\|_A<\varepsilon$. 
\end{itemize}

\end{definition}
\begin{definition}
As in \cite{KZ}, we let $\mathscr{D}_A$ be the set of all $\xi\in H_{\omega}$ which is represented by an $A$-sequence $(\xi_n)_n$ such that $(A\xi_n)_n$ is bounded, and let $H(A)=\overline{\mathscr{D}_A}$. We also define related subspaces:
define $\widehat{\mathscr{D}}_A$ to be the space of all $\xi\in H_{\omega}$ which is represented by a proper $A$-sequence. We also define $\mathscr{D}_0$ to be the set of all $\xi\in H_{\omega}$ which has a representative $(\xi_n)_n$ satisfying $\xi_n\in 1_{[-a,a]}(A)H$ for all $n\in \mathbb{N}$, where $a>0$ is a constant independent of $n$.\\
It is clear that $\mathscr{D}_0\subset \widehat{\mathscr{D}}_A\subset \mathscr{D}_A$.    
\end{definition}
The main result of the paper is that $\widehat{\mathscr{D}}_A=\mathscr{D}_A, K(A)=H(A)$ and $A_{\omega}=\widetilde{A}_{\omega}=\overline{A_{\omega}|_{\mathscr{D}_0}}$. 

In this section we will show that 
\begin{theorem}\label{thm: definition of A^{omega} is appropriate}$
\dom{\widetilde{A}}=\widehat{\mathscr{D}}_A\subset K(A)$, and $\mathscr{D}_0$ is a core for $\widetilde{A}_{\omega}$.
\end{theorem}
We need several lemmata. Next Lemma justifies the choice of proper $A$-sequences to consider the ultrapower.
\begin{lemma}\label{lem: hat{A}subset tilde{A}}
$\widehat{\mathscr{D}}_A\subset \dom{\widetilde{A}_{\omega}}$, and for $\xi\in \widehat{\mathscr{D}}_A$ with a proper representative $(\xi_n)_n$, we have
\[\widetilde{A}_{\omega}\xi=(A\xi_n)_{\omega}.\]
In particular, $(A\xi_n)_{\omega}=(A\xi_n')_{\omega}$ if both $(\xi_n)_n, (\xi_n')_n$ are proper $A$-sequences representing the same vector $\xi\in \widehat{\mathscr{D}}_A$. 
\end{lemma}
\begin{proof}
We first show that $\widehat{\mathscr{D}}_A\subset K(A)$. Since $K(A)$ is closed and every elements in $\widehat{\mathscr{D}}_A$ can be approximated by vectors of the form $(\eta_n)_{\omega}$ where $\eta_n\in 1_{[-a,a]}(A)H (n\in \mathbb{N})$ for a fixed $a>0$, it suffices to show that $\{t\mapsto e^{itA}\eta_n\}_{n=1}^{\infty}$ is $\omega$-equicontinuous for such $(\eta_n)_{\omega}$. Let $\varepsilon>0$ and $t\in \mathbb{R}$ be given. Let $A=\int_{\mathbb{R}}\lambda de(\lambda)$ be the spectral resolution of $A$. We have 
\eqa{
\|e^{itA}\eta_n-e^{isA}\eta_n\|^2&=\int_{\mathbb{R}}|e^{i(t-s)\lambda}-1|^2d\|e(\lambda)\eta_n\|^2\\
&=2\int_{\mathbb{R}}(1-\cos ((t-s)\lambda))d\|e(\lambda)\eta_n\|^2\\
&\le \int_{[-a,a]}(t-s)^2\lambda^2d\|e(\lambda)\eta_n\|^2\\
&\le (t-s)^2a^2\|\eta_n\|^2.
}
Therefore let $\delta>0$ be such that $\delta^2a^2\sup_{n\ge 1}\|\eta_n\|^2<\varepsilon^2$. Then for each $n\in \mathbb{N}$ and $s\in (t-\delta,t+\delta)$, $\|e^{itA}\eta_n-e^{isA}\eta_n\|<\varepsilon$ holds. Therefore $(\eta_n)_{\omega}$ is $A$-regular and $\widehat{\mathscr{D}}_A\subset K(A)$ holds.\\ \\
  
Next, let $\zeta:=(iA\xi_n)_{\omega}$. We show that $\frac{1}{t}(v(t)-1)\xi$ converges to $\zeta$ as $t\to 0$. Let $\varepsilon>0$. We may find $a>0$ and $(\eta_n)_n$ satisfying conditions in (*) of Definition \ref{def: proper A-sequences}. Let $\eta=(\eta_n)_{\omega}$. Then we have
\eqa{
\left \|\frac{1}{t}(v(t)-1)\xi-\zeta\right \|&\le \left \|\frac{1}{t}(v(t)-1)(\xi-\eta) \right \|+
\left \|\frac{1}{t}(v(t)-1)\eta-(iA\eta_n)_{\omega}\right \|\\
&\hspace{2.5cm}+\left \|(iA\eta_n)_{\omega}-(iA\xi_n)_{\omega}\right \|.
}
By the condition (*), the last term satisfies $\|(iA\eta_n)_{\omega}-(iA\xi_n)_{\omega}\|<\varepsilon$.\\ 
Now estimate the first term:
\eqa{
\left \|\frac{1}{t}(v(t)-1)(\xi-\eta)\right \|^2
&=\lim_{n\to \omega}\frac{1}{t^2}\int_{\mathbb{R}}|e^{it\lambda}-1|^2d\|e(\lambda)(\xi_n-\eta_n)\|^2\\
&\le \lim_{n\to \omega}\frac{1}{t^2}\int_{\mathbb{R}}t^2\lambda^2d\|e(\lambda)(\xi_n-\eta_n)\|^2\\
&=\|(A\xi_n)_{\omega}-(A\eta_n)_{\omega}\|^2<\varepsilon^2.
}
Using $\eta_n\in 1_{[-a,a]}(A)H\ (n\ge 1)$, we then estimate the second term:
\eqa{
\left \|\frac{1}{t}(v(t)-1)\eta-(iA\eta_n)_{\omega}\right \|^2&=\lim_{n\to \omega}\int_{-a}^a\left |\frac{e^{it\lambda}-1}{t}-i\lambda\right |^2d\|e(\lambda)\eta_n\|^2\\
&=\lim_{n\to \omega}\int_{-a}^a\left \{\left (\frac{\cos (t\lambda)-1}{t}\right )^2+\left (\frac{\sin (t\lambda)}{t}-\lambda \right )^2\right \}d\|e(\lambda)\eta_n\|^2\\
&=\lim_{n\to \omega}\int_{-a}^aF(t,\lambda)d\|e(\lambda)\eta_n\|^2,
}
where
\[F(t,\lambda)=\lambda^2\left (2\frac{1-\cos (t\lambda)}{(t\lambda)^2}-2\frac{\sin (t\lambda)}{t\lambda}+1\right ).\]
Therefore we have for each $t$ with $|t|a<\frac{\pi}{2}$,
\eqa{
\sup_{|\lambda|\le a}F(t,\lambda)&\le 2a^2\sup_{|\lambda|\le a}\left (1-\frac{\sin (t\lambda)}{t\lambda}\right )\\
&=2a^2\sup_{|x|\le |t|a}\left (1-\frac{\sin x}{x}\right )\\
&=2a^2\left (1-\frac{\sin (ta)}{ta}\right ).
}
Therefore for $|t|<\frac{\pi}{2a}$, 
\eqa{
\lim_{n\to \omega}\int_{-a}^aF(t,\lambda)d\|e(\lambda)\eta_n\|^2&\le \lim_{n\to \omega}\int_{-a}^{a}2a^2\left (1-\frac{\sin (ta)}{ta}\right )d\|e(\lambda)\eta_n\|^2\\
&=2a^2\left (1-\frac{\sin (ta)}{ta}\right )\|(\eta_n)_{\omega}\|^2\\
&\stackrel{t\to 0}{\to 0}.
}
Therefore we have
\[\overline{\lim_{t\to 0}}\left \|\frac{1}{t}(v(t)-1)\eta-(iA\eta_n)_{\omega}\right \|\le 2\varepsilon.\]
Since $\varepsilon>0$ is arbitrary, the claim is proved.
\end{proof}

Now we show that the order of integration and ultralimit can be interchanged for $\omega$-equicontinuous family  $\{F_n\colon \mathbb{R}\to H\}_{n=1}^{\infty}$ under some additional conditions.
\begin{lemma}\label{lem: integration and ultralimit}
Let $F_n\in C(\mathbb{R},H)\cap L^1(\mathbb{R},H)\ (n\in \mathbb{N})$ be a family of $H$-valued $\omega$-equicontinuous maps satisfying the following two conditions:
\begin{align}
\int_{\mathbb{R}}\sup_{n\ge 1}\|F_n(t)\|dt<\infty,\ \ \  &\sup_{n\ge 1}\|F_n(t)\|<\infty \ \ \ (t\in \mathbb{R}).\label{eq: integrability of F_omega}\\
\lim_{a\to \infty}\lim_{n\to \omega}&\int_{\mathbb{R}\setminus [-a,a]}\|F_n(t)\|dt=0 \label{eq: regularity of F_n}
\end{align}
Then we have
\[\left (\int_{\mathbb{R}}F_n(t)dt\right )_{\omega}=\int_{\mathbb{R}}(F_n(t))_{\omega}dt.\] 
\end{lemma}
\begin{remark}
Note that by the $\omega$-equicontinuity of $\{F_n\}_{n=1}^{\infty}$, $t\mapsto (F_n(t))_{\omega}$ is continuous. In particular, it is measurable.
\end{remark}
\begin{proof} 
By Eq. (\ref{eq: integrability of F_omega}), we have 
\[\int_{\mathbb{R}}(F_n(t))_{\omega}dt=\lim_{a\to \infty}\int_{-a}^a(F_n(t))_{\omega}dt.\]
By Eq. (\ref{eq: regularity of F_n}), we also have
\[\left (\int_{\mathbb{R}}F_n(t)dt\right )_{\omega}=\lim_{a\to \infty}\left (\int_{-a}^a F_n(t)dt\right )_{\omega}.\]
Therefore we have only to show $\int_{-a}^a(F_n(t))_{\omega}dt=\left (\int_{-a}^a F_n(t)dt\right )_{\omega}$ for all $a>0$. By the $\omega$-equicontinuity of $\{F_n\}_{n=1}^{\infty}$, there exists a partition $\Delta: t_0=-a<t_1<t_2<\cdots <t_N=a$ of the interval $[-a,a]$ and $W\in \omega$ so that for each $0\le i\le N-1$, $n\in W$ and $\alpha, \beta \in [t_i,t_{i+1}]$, we have 
\[\|F_n(\alpha)-F_n(\beta)\|<\varepsilon/4a.\]
This in particular implies that $\|(F_n(\alpha))_{\omega}-(F_n(\beta))_{\omega}\|<\varepsilon/4a$. Therefore by the definition of the Riemann integral, we have 
\[\left \| \sum_{i=0}^{N-1}(t_{i+1}-t_i)F_n(t_i)-\int_{-a}^aF_n(t)dt\right \| <\varepsilon/2, \ \ \ \ (n\in W),\]
and
\[\left \| \sum_{i=0}^{N-1}(t_{i+1}-t_i)(F_n(t_i))_{\omega}-\int_{-a}^a(F_n(t))_{\omega}dt\right \| <\varepsilon/2.\]
Using $\left (\sum_{i=0}^{N-1}(t_{i+1}-t_i)F_n(t_i)\right )_{\omega}=\sum_{i=0}^{N-1}(t_{i+1}-t_i)(F_n(t_i))_{\omega}$, we have
\[\left \| \int_{-a}^a(F_n(t))_{\omega}dt-\left (\int_{-a}^aF_n(t)dt\right )_{\omega}\right \|<\varepsilon.\]
Since $\varepsilon>0$ is arbitrary, the claim is proved.  
\end{proof}

\begin{lemma}\label{lem: smeared out function}
Let $\xi=(\xi_n)_{\omega}\in K(A)$, and let $f\in L^1(\mathbb{R})$. Then we have 
\[\left (\int_{\mathbb{R}}f(t)e^{itA}\xi_ndt\right )_{\omega}=\int_{\mathbb{R}}(f(t)e^{itA}\xi_n)_{\omega}dt.\]
\end{lemma}
\begin{proof}Note that $t\mapsto f(t)(e^{itA}\xi_n)_{\omega}$ is measurable thanks to Lemma \ref{lem: omega-equicontinuity is well-defined}. Let $C:=\sup_n\|\xi_n\|$. 
First assume that $f\in L^1(\mathbb{R})\cap C(\mathbb{R})$. 
It suffices to show that $\{F_n:t\mapsto f(t)e^{itA}\xi_n\}_{n=1}^{\infty}$ is $\omega$-equicontinuous and satisfies Eqs. (\ref{eq: integrability of F_omega}) and (\ref{eq: regularity of F_n}) in Lemma \ref{lem: integration and ultralimit}. It holds that
$\sup_n \int_{\mathbb{R}}\|F_n(t)\|dt=\int_{\mathbb{R}}|f(t)|dt\cdot \|\xi\|<\infty$, $\sup_n\|F_n(t)\|=|f(t)|<\infty$, and 
\[\lim_{a\to \infty}\lim_{n\to \omega}\int_{\mathbb{R}\setminus [-a,a]}\|F_n(t)\|dt=\lim_{a\to \infty}\int_{\mathbb{R}\setminus [-a,a]}|f(t)|dt\cdot \|\xi\|=0.\]
Therefore Eqs. (\ref{eq: integrability of F_omega}) and (\ref{eq: regularity of F_n}) in Lemma \ref{lem: integration and ultralimit} are satisfied. We show the $\omega$-equicontinuity of $\{F_n\}_{n=1}^{\infty}$. Suppose $\varepsilon>0$ and $t\in \mathbb{R}$ are given. By the $A$-regularity of $\xi$ and by the continuity of $f$, there exists $\delta>0$ and $W\in \omega$ such that for each $s\in (t-\delta,t+\delta)$ and $n\in W$, we have 
\[\|e^{itA}\xi_n-e^{isA}\xi_n\|<\frac{\varepsilon}{2(|f(t)|+1)},\ \ \ |f(t)-f(s)|<\frac{\varepsilon}{2(C+1)}.\]
It then follows that
\eqa{
\|f(t)e^{itA}\xi_n-f(s)e^{isA}\xi_n\|&\le |f(t)|\cdot \|e^{itA}\xi_n-e^{isA}\xi_n\|+|f(t)-f(s)|\cdot \|e^{isA}\xi_n\|\\
&<\varepsilon/2+\varepsilon/2=\varepsilon.
}
Therefore $\{F_n\}_{n=1}^{\infty}$ is $\omega$-equicontinuous. By Lemma \ref{lem: integration and ultralimit}, the claim follows.\\ \\
Next, suppose $f\in L^1(\mathbb{R})$. Let $\varepsilon>0$. There exists $g\in L^1(\mathbb{R})\cap C(\mathbb{R})$ such that $\|f-g\|_1<\varepsilon/2(C+1)$. Then we have 
\eqa{
\left \|\left (\int_{\mathbb{R}}f(t)e^{itA}\xi_ndt-\int_{\mathbb{R}}g(t)e^{itA}\xi_ndt\right )_{\omega}\right \|&\le \lim_{n\to \omega}\int_{\mathbb{R}}|f(t)-g(t)\||\xi_n\|dt\\
&< \varepsilon/2,\\
\left \|\int_{\mathbb{R}}(g(t)e^{itA}\xi_n)_{\omega}dt-\int_{\mathbb{R}}(f(t)e^{itA}\xi_n)_{\omega}dt\right \|&\le \|g-f\|_1\cdot \|\xi\|<\varepsilon/2, 
}
whence by applying the above argument for $g$ we have 
\[\left \|\left (\int_{\mathbb{R}}f(t)e^{itA}\xi_ndt\right )_{\omega}-\int_{\mathbb{R}}(f(t)e^{itA}\xi_n)_{\omega}dt\right \|<\varepsilon.\]
Since $\varepsilon>0$ is arbitrary, the claim is proved.
\end{proof}

\begin{lemma}\label{lem: D(tildeA)subset D(hatA)}
$\dom{\widetilde{A}_{\omega}}=\widehat{\mathscr{D}}_A$. 
\end{lemma}
\begin{proof}
By Lemma \ref{lem: hat{A}subset tilde{A}}, it suffices to show that $\dom{\widetilde{A}_{\omega}}\subset \widehat{\mathscr{D}}_A$.
Let $e(\cdot)\ (\text{resp.}\ \widetilde{e}(\cdot))$ be the spectral measure associated with $A\ (\text{resp.}\ \widetilde{A}_{\omega})$. We first show the following:\\ \\
\textbf{Claim.} For a given $\xi \in \dom{\widetilde{A}_{\omega}}$ and $\varepsilon>0$, there exists $a>0$ and $(\eta_n)_n\in \ell^{\infty}(\mathbb{N},H)$ with the properties: $\eta_n\in 1_{[-a,a]}(A)H\ (n\in \mathbb{N})$, \ $\|\xi-(\eta_n)_{\omega}\|<\varepsilon$ and $\|\widetilde{A}_{\omega}\xi-(A\eta_n)_{\omega}\|<\varepsilon$.\\ \\
We note that in general $\widetilde{e}(B)$ is not the ultrapower of $e(B)$ for a Borel set $B$. Therefore we need some extra work (cf. \cite[$\S$4]{AH}). As $\bigcup_{a>0}1_{[-a,a]}(\widetilde{A}_{\omega})K(A)$ is a core for $\widetilde{A}_{\omega}$, there exists $a>0$,  $\eta=(\eta_n)_{\omega}\in 1_{[-\frac{a}{2},\frac{a}{2}]}(\widetilde{A}_{\omega})K(A)$ such that $\|\xi-\eta\|<\varepsilon$ and $\|\widetilde{A}_{\omega}\xi-\widetilde{A}_{\omega}\eta\|<\varepsilon$. Let $f\in L^1(\mathbb{R})$ be a function with the following properties: $\text{supp}(\hat{f})\subset [-a,a]$, $\hat{f}=1$ on $[-\frac{a}{2},\frac{a}{2}]$, $0\le \hat{f}(\lambda)\le 1\ (\lambda \in \mathbb{R})$. Here, $\hat{f}(\lambda)=\int_{\mathbb{R}}e^{i\lambda t}f(t)dt$ is the Fourier transform of $f$. For instance, one may choose the de la Vall\'ee Poussin kernel $D_{a/2}$ (see \cite[Definition 4.12]{AH}).
Let 
\[\eta':=\int_{\mathbb{R}}f(t)e^{it\widetilde{A}_{\omega}}\eta dt.\]
We then have (by the spectral condition of $\eta$ and $\hat{f}=1$ on $[-a/2,a/2]$)
\eqa{
\eta'&=\int_{\mathbb{R}}\int_{\mathbb{R}}f(t)e^{it\lambda}d(\widetilde{e}(\lambda)\eta)dt=\int_{\mathbb{R}}\left (\int_{\mathbb{R}}f(t)e^{it\lambda}dt\right )d(\widetilde{e}(\lambda)\eta)\\
&=\int_{\mathbb{R}}\hat{f}(\lambda)d(\widetilde{e}(\lambda)\eta)=\hat{f}(\widetilde{A}_{\omega})\eta\\
&=\eta.
}
Furthermore, by Lemma \ref{lem: smeared out function}, we have
\[\eta=\eta'=\left (\int_{\mathbb{R}}f(t)e^{itA}\eta_ndt\right )_{\omega}=(\hat{f}(A)\eta_n)_{\omega},\]
and $\eta_n':=\hat{f}(A)\eta_n\in 1_{[-a,a]}(A)H$ for each $n\ge 1$. Therefore $(\eta_n')_n$ is the required sequence, as $\widetilde{A}_{\omega}\eta=(A\eta_n')_{\omega}$ (cf. Lemma \ref{lem: hat{A}subset tilde{A}}). \\ \\
Assume now that $\xi \in \dom{\widetilde{A}_{\omega}}$ with $\|\xi\|=1$. We show that $\xi\in \widehat{\mathscr{D}}_A$, i.e., it has a proper representative. Let $\varepsilon>0$. We use the following similar  argument to \cite[ Lemma 3.9 (i)]{AH}. By the above Claim, for each $k\in \mathbb{N}$, put $\varepsilon=2^{-k-1}$ in the above argument to find $a_k(\le a_{k+1}\le a_{k+2}\le \cdots)$ and $(\eta^{(k)}_n)_n\in \ell^{\infty}(\mathbb{N},H)$ satisfying $\eta_n^{(k)}\in 1_{[-a_k,a_k]}(A)H\ (n\in \mathbb{N})$, and  
\[\|\xi-(\eta^{(k)}_n)_{\omega}\|<\frac{1}{2^{k+1}},\ \ \ \|\widetilde{A}_{\omega}\xi-(A\eta^{(k)}_n)_{\omega}\|<\frac{1}{2^{k+1}},\ \ \ \ \ k\in \mathbb{N}.\]
Furthermore, we may assume $\|\eta_n^{(k)}\|\le 2$ for each $n,k\in \mathbb{N}$.
Then for each $k\in \mathbb{N}$, we have
\[
\|(\eta_n^{(k+1)})_{\omega}-(\eta_n^{(k)})_{\omega}\|<\frac{1}{2^k},\ \ \ \|(A\eta^{(k+1)}_n)_{\omega}-(A\eta^{(k)}_n)_{\omega}\|<\frac{1}{2^k}.\]
Let 
\[G_k:=\left \{n\in \mathbb{N}; \|\eta_n^{(k+1)}-\eta_n^{(k)}\|<\frac{1}{2^k},\ \ \|A\eta^{(k+1)}_n-A\eta^{(k)}_n\|<\frac{1}{2^k}\right \},\ \ k\in \mathbb{N}.\]
Then $G_k\in \omega\ (k\in \mathbb{N})$ holds, and since $\omega$ is free, $F_k:=\bigcap_{i=1}^kG_i\cap \{n\in \mathbb{N};n\ge k\}\in \omega \ (k\in \mathbb{N})$. Since $\{F_k\}_{k=1}^{\infty}$ is decreasing with empty intersection, $\mathbb{N}=(\mathbb{N}\setminus F_1)\sqcup \bigsqcup_{j=1}^{\infty}(F_j\setminus F_{j+1})$. Then define $(\xi_n)_n$ by
\[\xi_n:=\begin{cases}\eta_n^{(1)} & (n\in \mathbb{N}\setminus F_1)\\
\eta_n^{(k)} & (n\in F_k\setminus F_{k+1})\end{cases}.\]
Then $\sup_{n\ge 1}\|\xi_n\|\le 2<\infty$. Fix $k\ge 1$. If $n\in F_k=\bigsqcup_{j=k}^{\infty}(F_j\setminus F_{j+1})$, there is unique $j\ge k$ for which $n\in F_j\setminus F_{j+1}$ holds, so that $\xi_n=\eta_n^{(j)}$. Then we have
\eqa{
\|\xi_n-\eta_n^{(k)}\|&=\|\eta_n^{(j)}-\eta_n^{(k)}\|\le \sum_{i=k}^{j-1}\|\eta_n^{(i+1)}-\eta_n^{(i)}\|\\
&\le \sum_{i=k}^{j-1}\frac{1}{2^{i}}<\frac{1}{2^{k-1}},
}
so that $F_k\in \omega$ implies 
\[\|(\xi_n)_{\omega}-(\eta_n^{(k)})_{\omega}\|<\frac{1}{2^{k-1}},\ \ \ \ k\in \mathbb{N}.\]
Similarly, 
\[\|(A\xi_n)_{\omega}-(A\eta_n^{(k)})_{\omega}\|<\frac{1}{2^{k-1}},\ \ \ k\in \mathbb{N}.\]     
In particular, for each $k\in \mathbb{N}$ we have
\[\|\xi-(\xi_n)_{\omega}\|\le \|\xi-(\eta_n^{(k)})_{\omega}\|+\|(\eta_n^{(k)})_{\omega}-(\xi_n)_{\omega}\|<\frac{1}{2^{k-2}}.\]
Letting $k\to \infty$, we obtain $\xi=(\xi_n)_{\omega}$. We show that $(\xi_n)_n$ is a proper $A$-sequence. Suppose $\varepsilon>0$ is given. Take $k$ such that $\varepsilon>2^{-k+1}$, and put $a=a_k>0, \eta_n:=\eta_n^{(k)}$. Then by construction, $\eta_n\in 1_{[-a,a]}(A)H\ (n\in \mathbb{N})$, 
$\|(\xi_n)_{\omega}-(\eta_n)_{\omega}\|<\varepsilon$ and $\|(A\xi_n)_{\omega}-(A\eta_n)_{\omega}\|<\varepsilon$ holds. Therefore changing $\xi_n$ to be 0 if necessary for $n$ belonging to a set $I$ with $I \notin \omega$, we may assume that $(A\xi_n)_n$ is bounded, and $(\xi_n)_n$ is a proper $A$-sequence. This finishes the proof. 
\end{proof} 

\begin{lemma}\label{lem: resolvent of Atilde is UP of A}
Let $(\xi_n)_n\in \ell^{\infty}(\mathbb{N},H)$ be a seqence such that $\xi=(\xi_n)_{\omega} \in K(A)$. Then $(\widetilde{A}_{\omega}-i)^{-1}\xi=((A-i)^{-1}\xi_n)_{\omega}$ and $(\widetilde{A}_{\omega}+i)^{-1}\xi=((A+i)^{-1}\xi_n)_{\omega}$. 
\end{lemma}
\begin{proof}
Since $v(t)=e^{it\widetilde{A}}=(e^{itA})_{\omega}|_{K(A)}\ (t\in \mathbb{R})$, by the resolvent formula and by Lemma \ref{lem: smeared out function}, we have
\eqa{
(\widetilde{A}_{\omega}-i)^{-1}\xi&=i\int_0^{\infty}e^{-t}e^{-it\widetilde{A}_{\omega}}\xi dt=i\int_0^{\infty}e^{-t}(e^{-itA}\xi_n)_{\omega}dt\\
&=\left (i\int_0^{\infty}e^{-t}e^{-itA}\xi_ndt\right )_{\omega}=((A-i)^{-1}\xi_n)_{\omega}.
}
The latter identity follows similarly.
\end{proof}
\begin{remark}
Note that $(\widetilde{A}_{\omega}-i)^{-1}\xi=((A-i)^{-1}\xi_n)_{\omega}$ holds even if $(\xi_n)_n$ is not proper. The only requirement is $A$-regularity: $(\xi_n)_{\omega}\in K(A)$. 
\end{remark}

We are now ready to prove Theorem \ref{thm: definition of A^{omega} is appropriate}.
\begin{proof}[Proof of Theorem \ref{thm: definition of A^{omega} is appropriate}]
$\dom{\widetilde{A}_{\omega}}=\widehat{\mathscr{D}}_A$ is proved in Lemma \ref{lem: D(tildeA)subset D(hatA)}. 
Then for every $\xi\in \widehat{\mathscr{D}}_A$ and $\varepsilon>0$, there exists $\eta\in \mathscr{D}_0$ such that $\|\xi-\eta\|_{\widetilde{A}_{\omega}}<\varepsilon$ holds (cf. Lemma \ref{lem: hat{A}subset tilde{A}}). Therefore $\widetilde{A}_{\omega}$ is the closure of $\widetilde{A}_{\omega}|_{\mathscr{D}_0}$.
\end{proof}
\section{Alternative Description of $A_{\omega}$}
Now we are ready to show 
\begin{theorem}Under the same notations as in $\S$3, the following holds.
\begin{itemize}
\item[{\rm{(1)}}] $K(A)=H(A)$, and $A_{\omega}=\widetilde{A}_{\omega}$. Moreover, $\mathscr{D}_0$ is a core for $A_{\omega}$. 
\item[{\rm{(2)}}] For a representative $(\xi_n)_n$ of $\xi\in \dom{A_{\omega}}$, $A_{\omega}\xi=(A\xi_n)_{\omega}$ holds if and only if it is a proper $A$-sequence (see Definition \ref{def: proper A-sequences}).
\end{itemize}
\end{theorem}
\begin{proof} (1) 
By construction, it is clear that $\widetilde{A}_{\omega}$ is a p.u. of $A$ in $K(A)\subset H_{\omega}$. Therefore by the maximality of $A_{\omega}$, Theorem \ref{thm: KZ p.u.} (2), $K(A)\subset H(A)$ and $\widetilde{A}_{\omega}=A_{\omega}|_{K(A)}$. Therefore if we show that $K(A)=H(A)$, then $\widetilde{A}_{\omega}=A_{\omega}$ holds. To show $H(A)\subset K(A)$, suppose that $(\xi_n)_n$ is a representing sequence of $\xi\in \mathscr{D}_A$ with $(A\xi_n)_n\in \ell^{\infty}(\mathbb{N},H)$. We show that $\{f_n:t\mapsto e^{itA}\xi_n\}_{n=1}^{\infty}$ is $\omega$-equicontinuous. Let $C:=\sup_{n}\|A\xi_n\|$. Then for $t,s\in \mathbb{R}$, as in the analysis in $\S$3,
\[\|e^{itA}\xi_n-e^{isA}\xi_n\|^2=\int_{\mathbb{R}}|e^{it\lambda}-e^{is\lambda}|^2d\|e(\lambda)\xi_n\|^2\le (t-s)^2\|A\xi_n\|^2\le C^2(t-s)^2,\]
which tends to 0 as $(t-s)\to 0$ uniformly in $n$. Therefore $\{f_n\}_{n=1}^{\infty}$ is $\omega$-equicontinuous. Therefore $\mathscr{D}_A\subset K(A)$, and taking the closure, $H(A)\subset K(A)$ holds. Therefore $H(A)=K(A)$. By Theorem \ref{thm: definition of A^{omega} is appropriate}, $\mathscr{D}_0$ is a core for $A_{\omega}=\widetilde{A}_{\omega}$. \\
(2) This follows from (1), Theorem \ref{thm: definition of A^{omega} is appropriate}, Lemma \ref{lem: hat{A}subset tilde{A}} and a simple observation that if $A_{\omega}\xi=(A\xi_n)_{\omega}$ and if $(\xi_n')_n$ is another proper $A$-sequence representing $\xi$, then for every $\varepsilon>0$ there is $a>0$ and an $A$-sequence $(\eta_n)_n$ with $\eta_n\in 1_{[-a,a]}(A)H\ (n\in \mathbb{N})$ such that 
$\lim_{n\to \omega}\|\xi_n-\eta_n\|_A=\lim_{n\to \omega}\|\xi_n'-\eta_n\|_A<\varepsilon$, so that $(\xi_n)_n$ is proper as well. 
\end{proof}
\begin{remark}\label{rem: (Axi_n') is bad}
Finally, let us return to Example \ref{ex: (Axi_n) not equals (Axi_n')}. We note that $(\xi_n)_n$ is proper, while $(\xi_n')_n$ is not. The first claim is obvious. For the latter, if it were proper, then so is $(\frac{1}{n}\eta_n)_n$. But if $(\frac{1}{n}\eta_n)_n$ were proper, there exists an $A$-sequence $(\zeta_n)_n$ and $a>0$ for which $\zeta_n\in 1_{[-a,a]}(A)H\ (n\in \mathbb{N})$ and $\lim_{n\to \omega}\|\zeta_n-\frac{1}{n}\eta_n\|<1/2$ and $\lim_{n\to \omega}\|A\zeta_n-\eta_n\|<\frac{1}{2}$. Let $n_0\in \mathbb{N}$ such that $n_0>|a|$. Then for $n\ge n_0$, $\eta_n\in 1_{\{n\}}(A)H$, so $\eta_n\perp \zeta_n$. Thus
\[\lim_{n\to \omega}\|A\zeta_n-\eta_n\|^2=\lim_{n\to \omega}\|A\zeta_n\|^2+1<\frac{1}{4},\]
which is a contradiction. Thus $(\frac{1}{n}\eta_n)_n$, whence $(\xi_n')_n$, is not proper. Note also that $(\eta_n)_{\omega}$ is perpendicular to $H(A)$, and in particular $(A\xi_n')_{\omega}\notin H(A)$. To see this, let $(\xi_n)_n$ be an $A$-sequence such that $(A\xi_n)_n$ is bounded. Then 
\eqa{
\left |\lim_{n\to \omega}\nai{\eta_n}{\xi_n}\right |&=\left |\lim_{n\to \omega}\nai{\eta_n}{(A-i)^{-1}(A+i)\xi_n}\right |\\
&\le \lim_{n\to \omega}\frac{1}{|n+i|}\|(A+i)\xi_n\|=0.
}
Thus $(\eta_n)_{\omega}\in \mathscr{D}_A^{\perp}=H(A)^{\perp}$.
\end{remark}
\section*{Acknowledgement}
The current research started from the authors discussion during the workshop 15th Noncommutative Harmonic Analysis (September 23-29 2012) in Bedlewo. We thank the organizers for their hospitality extended to them. We also thank Yasumichi Matsuzawa for many useful comments which improved the presentation of the paper.   

Hiroshi Ando\\
Institut des Hautes \'Etudes Scientifiques,\\
Le Bois-Marie 35, route de Chartres,\\
91440 Bures-sur-Yvette, France\\
ando@ihes.fr\\ \\
Izumi Ojima\\
Research Institute for Mathematical Sciences,\\
Kyoto University,\\
Sakyo-ku, Kitashirakawa Oiwakecho\\
606-8502 Kyoto, Japan\\
ojima@kurims.kyoto-u.ac.jp\\ \\
Hayato Saigo\\
Nagahama Institute of Bio-Science and Technology, \\
Nagahama 526-0829, Japan\\
h\_saigoh@nagahama-i-bio.ac.jp
\end{document}